\documentclass{amsart}
\usepackage{amsfonts, amsbsy, amsmath, amssymb}

\newtheorem{thm}{Theorem}[section]

\newtheorem{fac}[thm]{Fact}
\numberwithin{equation}{section}

\theoremstyle{definition}

\begin{document}

\title[Proof of a Conjecture on Permutation Polynomials]{Proof of a Conjecture on Permutation Polynomials over Finite Fields}

\author[Xiang-dong Hou]{Xiang-dong Hou*}
\address{Department of Mathematics and Statistics,
University of South Florida, Tampa, FL 33620}
\email{xhou@usf.edu}
\thanks{* Research partially supported by NSA Grant H98230-12-1-0245.}

\keywords{exponential sum, finite field, permutation polynomial}

\subjclass[2000]{11T06, 11T55}

\begin{abstract}
Let $k$ be a positive integer and $S_{2k}={\tt x}+{\tt x}^4+\cdots+{\tt x}^{4^{2k-1}}\in\Bbb F_2[{\tt x}]$. It was recently conjectured that ${\tt x}+S_{2k}^{4^{2k}}+S_{2k}^{4^k+3}$ is a permutation polynomial of $\Bbb F_{4^{3k}}$. In this note, the conjecture is confirmed and a generalization is obtained. 
\end{abstract}

\maketitle

\section{Introduction}

Let $\Bbb F_q$ denote the finite fields with $q$ elements. For integer $k\ge 0$, define
\[
S_{k,q}={\tt x}+{\tt x}^q+\cdots+{\tt x}^{q^{k-1}}\in\Bbb F_p[{\tt x}]\qquad (p=\text{char}\,\Bbb F_q).
\]
When $q$ is fixed, we write $S_{k,q}=S_k$. The main purpose of this note is to prove the following theorem.

\begin{thm}\label{T1.1}
Let $q=4$ and $k\ge 1$. Then ${\tt x}+S_{2k}^{q^{2k}}+S_{2k}^{q^k+3}$ is a permutation polynomial (PP) of $\Bbb F_{q^{3k}}$.
\end{thm}

Theorem~\ref{T1.1} appeared as a conjecture in a recent study of permutation polynomials over finite fields \cite{FHL}. We refer the reader to \cite{FHL} for more background of this conjecture. An adaptation of the proof of Theorem~\ref{T1.1} gives a slight generalization of Theorem~\ref{T1.1}, which is also included in thos note.

\section{Proof of Theorem~\ref{T1.1}}

We first recall two facts:

\begin{fac}\cite[Theorem~7.7]{LN}\label{F2.1}
\rm
Let $f\in\Bbb F_q[{\tt x}]$. Then $f$ is a PP of $\Bbb F_q$ if and only if for all $a\in\Bbb F_q^*$, $\sum_{x\in\Bbb F_q}\zeta_p^{\text{\rm Tr}_{q/p}(af(x))}=0$, where $p=\text{\rm char}\,\Bbb F_q$ and $\zeta_p=e^{\frac{2\pi i}p}$. 
\end{fac}

\begin{fac}\cite[Lemma~6.13]{FHL}\label{F2.2}
\rm
Let $p$ be a prime and $f:\Bbb F_p^n\to \Bbb F_p$ a function. If there exists a $y\in\Bbb F_p^n$ such that $f(x+y)-f(x)$ is a nonzero constant for all $x\in\Bbb F_p^n$, then $\sum_{x\in\Bbb F_p^n}\zeta_p^{f(x)}=0$.
\end{fac}

\begin{proof}[Proof of Theorem~\ref{T1.1}] 
Recall that $q=4$ and $k\ge 1$. Write $g={\tt x}+S_{2k}^{q^{2k}}+S_{2k}^{q^k+3}$ and $\text{Tr}=\text{Tr}_{q^{3k}/2}$. By Fact~\ref{F2.1}, it suffices to show that for every $a\in \Bbb F_{q^{3k}}^*$, 
\begin{equation}\label{2.1}
\sum_{x\in\Bbb F_{q^{3k}}}(-1)^{\text{Tr}(ag(x))}=0.
\end{equation}
We will use the relation 
\begin{equation}\label{2.2}
S_{2k}+S_{2k}^{q^k}+S_{2k}^{q^{2k}}\equiv 0\pmod{{\tt x}^{q^{3k}}-{\tt x}}.
\end{equation}

\medskip
{\bf Case 1.}
Assume $\text{Tr}_{q^{3k}/q^k}(a)\ne 0$. Then there exists a $y\in\Bbb F_{q^k}$ such that $\text{Tr}_{q^k/2}\bigl[y\text{Tr}_{q^{3k}/q^k}(a)\bigr]\ne 0$. For all $x\in\Bbb F_{q^{3k}}$ we have $S_{2k}(x+y)=S_{2k}(x)+S_{2k}(y)=S_{2k}(x)$. Hence
\[
\text{Tr}\bigl[ag(x+y)-ag(x)\bigr]=\text{Tr}(ay)=\text{Tr}_{q^k/2}\bigl[y\text{Tr}_{q^{3k}/q^k}(a)\bigr],
\]
which is a nonzero constant. By Fact~\ref{F2.2}, \eqref{2.1} holds.

\medskip
{\bf Case 2.}
Assume $\text{Tr}_{q^{3k}/q^k}(a)= 0$. The $a=c+c^{q^k}$ for some $c\in\Bbb F_{q^{3k}}\setminus\Bbb F_{q^k}$. For $x\in \Bbb F_{q^{3k}}$, we write $S_{2k}(x)=S_{2k}$ (a slight abuse of notation). We have 
\begin{equation}\label{2.3}
\begin{split}
&\text{Tr}\bigl(ag(x)\bigr)\cr
=\,&\text{Tr}\bigl[(c+c^{q^k})g(x)\bigr]\cr
=\,&\text{Tr}\bigl[c\bigl(g(x)+g(x)^{q^{2k}}\bigr)\bigr]\cr
=\,&\text{Tr}\bigl[c(x+S_{2k}^{q^{2k}}+S_{2k}^{q^k+3}+x^{q^{2k}}+S_{2k}^{q^k}+S_{2k}^{1+3q^{2k}})\bigr]\cr
=\,&\text{Tr}\bigl[c(x+x^{q^{2k}}+S_{2k}+S_{2k}^3(S_{2k}+S_{2k}^{q^{2k}})+S_{2k}^{1+3q^{2k}})\bigr]\kern 1cm\text{(by \eqref{2.2})}\cr
=\,&\text{Tr}\bigl[c(S_{2k}^q+S_{2k}^4+S_{2k}^{3+q^{2k}}+S_{2k}^{1+3q^{2k}})\bigr]\cr
=\,&\text{Tr}\bigl[cS_{2k}^{1+q^{2k}}(S_{2k}^2+S_{2k}^{2q^{2k}})\bigr]\cr
=\,&\text{Tr}\bigl[cS_{2k}^{1+q^{2k}}(S_{2k}+S_{2k}^{q^{2k}})^2\bigr]\cr
=\,&\text{Tr}(cS_{2k}^{1+q^{2k}}S_{2k}^{2q^k}) \kern 6.3cm\text{(by \eqref{2.2})}\cr
=\,&\text{Tr}(cS_{2k}^{1+2q^k+q^{2k}}).
\end{split}
\end{equation}
By \eqref{2.1}, $S_{2k}(\Bbb F_{q^{2k}})\subset\text{Tr}_{q^{3k}/q^k}^{-1}(0)$. Since 
\[
\text{gcd}(1+{\tt x}+\cdots+{\tt x}^{2k-1},{\tt x}^{3k}-1)={\tt x}^k-1,
\]
by \cite[Theorem~3.62]{LN}, $\{x\in\Bbb F_{q^{3k}}:S_{2k}(x)=0\}=\Bbb F_{q^k}$. Thus $S_{2k}:\Bbb F_{q^{3k}}\to\text{Tr}_{q^{3k}/q^k}^{-1}(0)$ is an onto $\Bbb F_{q^k}$-map with $\ker\,S_{2k}=\Bbb F_{q^k}$. Therefore
\begin{equation}\label{2.4}
\begin{split}
\sum_{x\in\Bbb F_{q^{3k}}}(-1)^{\text{Tr}(ag(x))}\,&=\sum_{x\in\Bbb F_{q^{3k}}}(-1)^{\text{Tr}(cS_{2k}^{1+2q^k+q^{2k}})}\kern1.5cm \text{(by \eqref{2.3})}\cr
&=q^k\sum_{x\in \text{Tr}_{q^{3k}/q^k}^{-1}(0)}(-1)^{\text{Tr}(cx^{1+2q^k+q^{2k}})}.
\end{split}
\end{equation}
Let $d_1,d_2$ be a basis of $\text{Tr}_{q^{3k}/q^k}^{-1}(0)$ over $\Bbb F_{q^k}$. We have
\begin{equation}\label{2.5}
\begin{split}
&\sum_{x\in \text{Tr}_{q^{3k}/q^k}^{-1}(0)}(-1)^{\text{Tr}(cx^{1+2q^k+q^{2k}})}\cr
=\,&\sum_{x,y\in\Bbb F_{q^k}}(-1)^{\text{Tr}[c(d_1x+d_2y)^{1+2q^k+q^{2k}}]}\cr
=\,&\sum_{x\in\Bbb F_{q^k}}(-1)^{\text{Tr}[c(d_1x)^{1+2q^k+q^{2k}}]}+\sum_{\substack{x\in\Bbb F_{q^k}\cr y\in\Bbb F_{q^k}^*}}(-1)^{\text{Tr}[c(d_1x+d_2y)^{1+2q^k+q^{2k}}]}\cr
=\,&\sum_{x\in\Bbb F_{q^k}}(-1)^{\text{Tr}(cd_1^{q^k}x)}+\sum_{\substack{x\in\Bbb F_{q^k}\cr y\in\Bbb F_{q^k}^*}}(-1)^{\text{Tr}[c(d_1xy+d_2y)^{1+2q^k+q^{2k}}]}\cr
&\text{(for the first sum, $x \leftarrow d_1^{1+q^k+q^{2k}}x^{1+2q^k+q^{2k}}$; for the second sum $x\rightarrow xy$)}\cr
=\,&\sum_{x\in\Bbb F_{q^k}}(-1)^{\text{Tr}(cd_1^{q^k}x)}+\sum_{\substack{x\in\Bbb F_{q^k}\cr y\in\Bbb F_{q^k}^*}}(-1)^{\text{Tr}[c(d_1x+d_2)^{q^k}y]}\cr
&\text{(for the first sum, $y \leftarrow (d_1x+d_2)^{1+q^k+q^{2k}}y^{1+2q^k+q^{2k}}$)}\cr
=\,&\sum_{x\in\Bbb F_{q^k}}(-1)^{\text{Tr}(cd_1^{q^k}x)}+\sum_{\substack{x\in\Bbb F_{q^k}\cr y\in\Bbb F_{q^k}^*}}(-1)^{\text{Tr}(cd_1^{q^k}xy+cd_2^{q^k}y)}\cr
=\,&\sum_{x\in\Bbb F_{q^k}}(-1)^{\text{Tr}(cd_1^{q^k}x)}+\sum_{\substack{x\in\Bbb F_{q^k}\cr y\in\Bbb F_{q^k}^*}}(-1)^{\text{Tr}(cd_1^{q^k}x+cd_2^{q^k}y)}\cr
&\text{(for the second sum, $x \leftarrow xy$)}\cr
=\,&\biggl[\sum_{x\in\Bbb F_{q^k}}(-1)^{\text{Tr}(cd_1^{q^k}x)}\biggr]\biggl[1+\sum_{y\in\Bbb F_{q^k}^*}(-1)^{\text{Tr}(cd_2^{q^k}y)}\biggr]\cr
=\,&\biggl[\sum_{x\in\Bbb F_{q^k}}(-1)^{\text{Tr}(cd_1^{q^k}x)}\biggr]\biggl[\sum_{y\in\Bbb F_{q^k}}(-1)^{\text{Tr}(cd_2^{q^k}y)}\biggr].
\end{split}
\end{equation}
We claim that $\text{Tr}_{q^{3k}/q^k}(cd_1^{q^k})$ and $\text{Tr}_{q^{3k}/q^k}(cd_2^{q^k})$ cannot be both $0$. (Otherwise, $\text{Tr}_{q^{3k}/q^k}^{-1}(0)=d_1^{q^k}\Bbb F_{q^k}+d_2^{q^k}\Bbb F_{q^k}$ would be a vector space over $\Bbb F_{q^k}(c)=\Bbb F_{q^{3k}}$, which is impossible.) Therefore, at least one of the two sums in the last line of \eqref{2.3} is $0$, so we have
\begin{equation}\label{2.6}
\sum_{x\in \text{Tr}_{q^{3k}/q^k}^{-1}(0)}(-1)^{\text{Tr}(cx^{1+2q^k+q^{2k}})}=0.
\end{equation}
Combining \eqref{2.4} and \eqref{2.6} gives \eqref{2.1}.
\end{proof}

\section{Generalization}

The proof in Section 3 works for the following generalization of Theorem~\ref{T1.1}

\begin{thm}\label{T3.1}
Let $q$ be a power of $2$. Let $L\in\Bbb F_{q^{3k}}[{\tt x}]$ be a $2$-linearized polynomial such that 
\begin{itemize}
  \item [(i)] $L$ permutes $\Bbb F_{q^k}$, and
  \item [(ii)] $L+L^{q^{2k}}\equiv S_{2k}^4\pmod{{\tt x}^{q^{3k}}-{\tt x}}$.
\end{itemize}  
Then $L+S_{2k}^{q^k+3}$ is a PP of $\Bbb F_{q^{3k}}$.
\end{thm}

\noindent{\bf Note.} 
$L=({\tt x}+S_{2k}^{q^{2k}})^{4q^{3k-1}}$ satisfies conditions (i) and (ii) of Theorem~\ref{T3.1}. (i) is obvious. For (ii), we have 
$L+L^{q^{2k}}=({\tt x}+S_{2k}^{q^{2k}}+{\tt x}^{q^{2k}}+S_{2k}^{q^k})^{4q^{3k-1}}\equiv ({\tt x}+{\tt x}^{q^{2k}}+S_{2k})^{4q^{3k-1}}=(S_{2k}^q)^{4q^{3k-1}}\equiv S_{2k}^4$, where ``$\equiv$'' means ``$\equiv\pmod{{\tt x}^{q^{3k}}-{\tt x}}$''.

\begin{proof}[Proof of Theorem~\ref{T3.1}]
Let $g=L+S_{2k}^{q^k+3}$ and follow the proof of Theorem~\ref{T1.1}. In Case 1, we choose $y\in\Bbb F_{q^k}$ such that $\text{Tr}_{q^k/2}\bigl[L(y)\text{Tr}_{q^{3k}/q^k}(a)\bigr]\ne 0$. In Case 2, \eqref{2.3}, still holds because of condition (ii).
\end{proof} 


\end{document}